\documentclass[11pt]{amsart}
\usepackage{geometry}                % See geometry.pdf to learn the layout options. There are lots.
\usepackage{graphicx}
\usepackage{amssymb}
\usepackage{epstopdf}
\usepackage{comment}
\usepackage{color}
\usepackage{pstricks, pst-node, pst-tree}
\graphicspath{{Figures/}}

\title[S\l upecki digraphs]
{S\l upecki digraphs}

\author[\'A. Kunos]{\'Ad\'am Kunos}
\address{Bolyai Institute, University of Szeged}
\email{akunos@math.u-szeged.hu}
\author[B. Larose]{Beno\^it Larose}
\address{Laboratoire d'Alg\`ebre, de Combinatoire et d'Informatique Math\'ematique (LACIM) \\
Universit\'e du Qu\'ebec \`a Montr\'eal}
\email{blarose@crcmail.net}
\author[D. E. Pazmi\~no P.]{David Emmanuel Pazmi\~no Pullas}
\address{Laboratoire d'Alg\`ebre, de Combinatoire et d'Informatique Math\'ematique (LACIM) \\
Universit\'e du Qu\'ebec \`a Montr\'eal}
\email{pazmino\_pullas.david\_emmanuel@courrier.uqam.ca}

\thanks{We wish to thank Barnaby Martin for his very helpful comments.}

\date{\today}
\begin{document}
\newtheorem{dummy}{Dummy}[section]

\newtheorem{conj}[dummy]{Conjecture}
\newtheorem{lemma}[dummy]{Lemma}
\newtheorem{prop}[dummy]{Proposition}
\newtheorem{theorem}[dummy]{Theorem}
\newtheorem{corollary}[dummy]{Corollary}
\newtheorem{definition}[dummy]{Definition}
\newtheorem{example}[dummy]{Example}

%SHORTCUTS FOR MATHCAL

\newcommand{\cA}{\mathcal{A}}
\newcommand{\cB}{\mathcal{B}}
\newcommand{\cC}{\mathcal{C}}
\newcommand{\cD}{\mathcal{D}}
\newcommand{\cE}{\mathcal{E}}
\newcommand{\cF}{\mathcal{F}}
\newcommand{\cG}{\mathcal{G}}
\newcommand{\cH}{\mathcal{H}}
\newcommand{\cI}{\mathcal{I}}
\newcommand{\cJ}{\mathcal{J}}
\newcommand{\cK}{\mathcal{K}}
\newcommand{\cL}{\mathcal{L}}
\newcommand{\cM}{\mathcal{M}}
\newcommand{\cN}{\mathcal{N}}
\newcommand{\cO}{\mathcal{O}}
\newcommand{\cP}{\mathcal{P}}
\newcommand{\cQ}{\mathcal{Q}}
\newcommand{\cR}{\mathcal{R}}
\newcommand{\cS}{\mathcal{S}}
\newcommand{\cT}{\mathcal{T}}
\newcommand{\cU}{\mathcal{U}}
\newcommand{\cV}{\mathcal{V}}
\newcommand{\cW}{\mathcal{W}}
\newcommand{\cX}{\mathcal{X}}
\newcommand{\cY}{\mathcal{Y}}
\newcommand{\cZ}{\mathcal{Z}}

%SHORTCUTS FOR MATHBF

\newcommand{\bA}{\mathbf{A}}
\newcommand{\bB}{\mathbf{B}}
\newcommand{\bC}{\mathbf{C}}
\newcommand{\bD}{\mathbf{D}}
\newcommand{\bE}{\mathbf{E}}
\newcommand{\bF}{\mathbf{F}}
\newcommand{\bG}{\mathbf{G}}
\newcommand{\bH}{\mathbf{H}}
\newcommand{\bI}{\mathbf{I}}
\newcommand{\bJ}{\mathbf{J}}
\newcommand{\bK}{\mathbf{K}}
\newcommand{\bL}{\mathbf{L}}
\newcommand{\bM}{\mathbf{M}}
\newcommand{\bN}{\mathbf{N}}
\newcommand{\bO}{\mathbf{O}}
\newcommand{\bP}{\mathbf{P}}
\newcommand{\bQ}{\mathbf{Q}}
\newcommand{\bR}{\mathbf{R}}
\newcommand{\bS}{\mathbf{S}}
\newcommand{\bT}{\mathbf{T}}
\newcommand{\bU}{\mathbf{U}}
\newcommand{\bV}{\mathbf{V}}
\newcommand{\bW}{\mathbf{W}}
\newcommand{\bX}{\mathbf{X}}
\newcommand{\bY}{\mathbf{Y}}
\newcommand{\bZ}{\mathbf{Z}}
\newcommand{\bGS}{{\bf S}}

%SHORTCUTS FOR MATHbb

\newcommand{\bbA}{\mathbb{A}}
\newcommand{\bbB}{\mathbb{B}}
\newcommand{\bbC}{\mathbb{C}}
\newcommand{\bbD}{\mathbb{D}}
\newcommand{\bbE}{\mathbb{E}}
\newcommand{\bbF}{\mathbb{F}}
\newcommand{\bbG}{\mathbb{G}}
\newcommand{\bbH}{\mathbb{H}}
\newcommand{\bbI}{\mathbb{I}}
\newcommand{\bbJ}{\mathbb{J}}
\newcommand{\bbK}{\mathbb{K}}
\newcommand{\bbL}{\mathbb{L}}
\newcommand{\bbM}{\mathbb{M}}
\newcommand{\bbN}{\mathbb{N}}
\newcommand{\bbO}{\mathbb{O}}
\newcommand{\bbP}{\mathbb{P}}
\newcommand{\bbQ}{\mathbb{Q}}
\newcommand{\bbR}{\mathbb{R}}
\newcommand{\bbS}{\mathbb{S}}
\newcommand{\bbT}{\mathbb{T}}
\newcommand{\bbU}{\mathbb{U}}
\newcommand{\bbV}{\mathbb{V}}
\newcommand{\bbW}{\mathbb{W}}
\newcommand{\bbX}{\mathbb{X}}
\newcommand{\bbY}{\mathbb{Y}}
\newcommand{\bbZ}{\mathbb{Z}}

\newcommand{\anote}[1]{{\color{purple} [{\bf Adam note:} #1]}}
\newcommand{\bnote}[1]{{\color{red} [{\bf Benoit note:} #1]}}
\newcommand{\dnote}[1]{{\color{blue} [{\bf David note:} #1]}}

\newcommand{\csp}{\mathsf{CSP}}

\begin{abstract}  Call a finite relational structure {\em $k$-S\l upecki} if its only surjective $k$-ary polymorphisms are essentially unary, and {\em S\l upecki} if it is  $k$-S\l upecki for all $k \geq 2$. We present conditions, some necessary and some sufficient, for a reflexive digraph to be S\l upecki. We prove that all digraphs that triangulate a 1-sphere are S\l upecki, as are all the ordinal sums $m \oplus n$ ($m,n \geq 2$). We prove that the posets $\bbP = m \oplus n \oplus k$ are not 3-S\l upecki for $m,n,k \geq 2$, and prove there is a bound $B(m,k)$ such that $\bbP$ is 2-S\l upecki if and only if $n > B(m,k)+1$; in particular there exist posets that are 2-S\l upecki but not 3-S\l upecki.   \end{abstract}

\maketitle
%\section{}
%\subsection{}

%\acks{We wish to thank many people for their help.}

%\bnote{THIS IS A DRAFT DO NOT DISTRIBUTE}

\section{Introduction}  

%\begin{center} {\bf In this paper, all digraphs are assumed to be reflexive.} \end{center}  \mbox{}\\

For our purposes, a digraph is a finite, non-empty set $A$ together with a binary relation on $A$.  {\bf In this paper, all digraphs are assumed to be reflexive}, i.e. that the binary relation contains $(a,a)$ for all $a \in A$. We study the surjective polymorphisms of these objects; these operations play an important role in the complexity of the related quantified constraint satisfaction problem (QCSP) and the surjective H-colouring problem, as do idempotent polymorphisms in the study of the associated CSP (see for instance \cite{MR3631056, laretal, larmarpau, DBLP:conf/dagstuhl/Martin17}). 

A digraph is $k$-{\em idempotent trivial} if all its $k$-ary idempotent polymorphisms are projections, and {\em idempotent trivial} if it is $k$-idempotent trivial for all $k$. Similarly, we say a digraph is $k$-{\em S\l upecki} if all its $k$-ary surjective polymorphisms are essentially unary, and {\em S\l upecki} if it is $k$-S\l upecki for all $k$. Clearly  a digraph is idempotent trivial if it is S\l upecki. It is known that these properties respectively imply NP-hardness and PSPACE-hardness of the CSP and QCSP naturally associated to the digraph (see \cite{MR3631056} and \cite{DBLP:conf/dagstuhl/Martin17}). The S\l upecki property has also been studied independently, see for example \cite{MR1020459} and \cite{MR2480633}.

The starting point of our investigation was to find some workable sufficient condition for a digraph to be S\l upecki: in particular, we wanted to verify that all cycles of girth at least 4 have this property.  There is a known sufficient condition for a digraph to be idempotent trivial that relies on the homotopy of a simplicial complex naturally associated to the digraph (see section \ref{sect-prelim}): if for some $n>0$ the $n$-th homotopy group of the space is non-trivial, but the $n$-th homotopy group of any proper retract is trivial, and provided the identity is isolated in the digraph of endomorphisms, then the digraph is idempotent trivial \cite{MR2232298}. It turns out that digraphs that triangulate $n$-spheres satisfy these properties, and thus are potentially S\l upecki. At first glance, the most straightforward approach to show a structure is S\l upecki is to pp-define the S\l upecki relation from its basic relations (see Section \ref{sect-gadgets}), and we manage to do this for a variety of digraphs triangulating $n$-spheres; however, this approach seems impractical for general cycles. In section \ref{sect-conditions} we present an alternative sufficient condition: provided the identity is isolated in the digraph of endomorphisms, if every onto polymorphism of the idempotent trivial digraph is a retraction, then the digraph is S\l upecki (see Lemma \ref{lemma-embedding} and Theorem \ref{thm-embedding}). In a companion paper, this result is used to prove that cycles of girth at least 4 are indeed S\l upecki \cite{LLP}.  The slightly annoying technical condition on the identity endomorphism can be removed for several families of digraphs (Lemmas \ref{629952} and \ref{lemma-not-strong-id-alone}), but we provide an example of a S\l upecki digraph that does not satisfy it (Lemma \ref{lemma-example}). 
However, we do not know if the condition can be removed from the statement of Lemma \ref{lemma-embedding}. 

As mentioned above, examining small digraphs triangulating spheres, it turns out that many are S\l upecki, such as for instance, all digraphs that triangulate a 1-sphere (Theorem \ref{theorem-big-paper}), but in section \ref{sect-sums} we provide an example of a small digraph triangulating a 2-sphere which isn't S\l upecki: indeed, the poset $2 \oplus 2 \oplus 2$ satisfies the conditions of Theorem \ref{thm-not-slu}, and is the suspension of a 4-cycle. Furthermore, we provide examples of posets that are 2-S\l upecki but not 3-S\l upecki (Theorem \ref{757370}); contrast this with the fact that all digraphs that are 2-idempotent trivial are idempotent trivial (see \cite{prs}). 

We now outline the contents of the paper. In section \ref{sect-prelim} we introduce basic terminology and notation. In section \ref{sect-conditions} we present various sufficient and necessary conditions for a digraph to be S\l upecki. The main result of that section, Theorem \ref{thm-embedding}, is a crucial tool for the main result of the paper \cite{LLP}, namely that all cycles of girth at least 4 are S\l upecki, and is also used in section \ref{sect-sums} to provide some examples of S\l upecki posets.  In section \ref{sect-gadgets} we prove that various digraphs that triangulate spheres are indeed S\l upecki, via pp-definitions using gadgets. We also state the result that all digraphs that triangulate 1-spheres are S\l upecki (Theorem \ref{theorem-big-paper}). In section \ref{sect-sums} we consider posets that are ordinal sums of antichains; we exhibit small digraphs that triangulate spheres but are not S\l upecki (Theorem \ref{thm-not-slu}); we also provide examples of posets that are 3-S\l upecki but not 2-S\l upecki Theorem \ref{757370}. In section \ref{sect-conclusion} we discuss various open questions that follow naturally from our results.

\section{Preliminaries: notation, definitions, etc.} \label{sect-prelim}

\subsection{Reflexive digraphs}

A {\em digraph} $\bbG = \langle G;E \rangle$ consists of a non-empty set  $G$ of {\em vertices} and  a binary relation $E$ on $G$; the pairs in $E$ are called the {\em arcs} or {\em edges} of $\bbG$.  It is {\em reflexive} if $(x,x) \in E$ for all $x \in E$; an arc of the form $(x,x)$ we call a {\em loop}. When we consider digraphs $\bbG$, $\bbH$, etc. we denote their set of vertices by $G$, $H$, etc. We sometimes write $u \rightarrow v$ to mean that $(u,v)$ is an arc of a digraph. Consider the undirected graph $\bbG_{Sym}$ obtained from the digraph $\bbG$ as follows: it has the same set of vertices $G$, and two vertices $x$ and $y$ are adjacent if one of $(x,y)$ or $(y,x)$ is an arc of $\bbG$. We say $\bbG$ is {\em connected} if $\bbG_{Sym}$ is connected.  We say the digraph $\bbG$ is {\em symmetric} if $x\rightarrow y$ implies $y \rightarrow x$ for all $x,y \in G$ (i.e. $\bbG$ is undirected). Let $\bbG$ be a digraph and $x,y \in G$. We say that $x$ and $y$ are in the same strong component of $\bbG$ if there exist $x=x_0, x_1,\cdots,x_t=y$ and $y=y_0, y_1,\cdots,y_s=x$ such that $x_i \rightarrow x_{i+1}$ for all $0 \leq i \leq t-1$ and  $y_j \rightarrow y_{j+1}$ for all $0 \leq j\leq s-1$. Clearly every digraph is partitioned into strongly connected components. We say $\bbG$ is {\em strongly connected} if it has exactly one strong component.  We say that a digraph $\bbH$ is {\em embedded} in a digraph $\bbG$ if it is isomorphic to an induced subdigraph of $\bbG$.

Let $\bbG$ and $\bbH$ be digraphs. A map $f:G \rightarrow H$ is a {\em homomorphism} if it preserves arcs, i.e. if $(x,y)$ is an arc of $\bbG$ then $(f(x),f(y))$ is an arc of $\bbH$.   If $\bbG$ and $\bbH$ are digraphs and we write $f:\bbH \rightarrow \bbG$ it is understood that $f$ is a homomorphism. We say  $f:\bbH \rightarrow \bbG$ is an {\em embedding} if it is an isomorphism onto its image (and hence $\bbH$ is embedded in $\bbG$). The {\em product} $\bbG \times \bbH$ of two digraphs $\bbG$ and $\bbH$ is the usual product of relational structures, i.e.  the digraph with set of vertices $G \times H$ and arcs $((g_1,h_1),(g_2,h_2))$ where $(g_1,g_2)$ and $(h_1,h_2)$ are arcs of $\bbG$ and $\bbH$ respectively; notice that $\bbG \times \bbH$ is reflexive if both $\bbG$ and $\bbH$ are reflexive. For every positive integer $k$, the product of $k$ digraphs is defined is the obvious way;  $\bbG^k$ is the product of $\bbG$ with itself $k$ times. A $k$-ary {\em polymorphism} of $\bbG$ is a homomorphism from $\bbG^k$ to $\bbG$; the integer $k$ is the {\em arity} of $f$; $f$ is {\em idempotent} if $f(x,\dots,x)=x$ for all $x \in G$. It is {\em essentially unary} if there exists some $1 \leq i \leq k$ and a homomorphism $g:\bbG \rightarrow \bbG$ such that $f(x_1,\dots,x_k) = g(x_i)$ for all $x_j \in G$; it is a {\em projection} if furthermore the homomorphism $g$ is the identity. 

\begin{definition}   Let $k \geq 2$. We say that the digraph $\bbG$ is {\em $k$-idempotent trivial} if all its idempotent $k$-ary polymorphisms are projections. We say that the digraph $\bbG$ is {\em  idempotent trivial} if it is $k$-idempotent trivial for all $k \geq 2$. \end{definition}

\begin{definition} Let $k \geq 2$. We say that the digraph $\bbG$ is {\em $k$-S\l upecki} if all its surjective $k$-ary polymorphisms are essentially unary. We say the digraph $\bbG$ is {\em S\l upecki} if it is $k$-S\l upecki for all $k \geq 2$. \end{definition}

It is easy to see that if a digraph is $k$-S\l upecki for some $k \geq 3$ then it is $(k-1)$-S\l upecki. Note also that a S\l upecki digraph is idempotent trivial.

\begin{definition}
Let $k \geq 0$.  A {\em path of length $k$} is a  digraph with vertex set $\{0,1,\dots,k\}$ where for each $0 \leq i \leq k-1$, one or both of the arcs $\{(i,i+1),(i+1,i)\}$ is 
present, and there are no other arcs.  \end{definition} 

\begin{definition} Let $n \geq 3$. An  {\em $n$-cycle} is a  digraph with vertex set $\{0,1,\dots,n-1\}$ where for each $0 \leq i < n-1$, one or both of the arcs $\{(i,i+1),(i+1,i)\}$ is present,  one or both of the arcs $\{(n-1,0),(0,n-1)\}$ is present, and there are no other arcs. The integer $n$ is called the {\em girth} of the cycle.  \end{definition}

We now discuss briefly the connection of reflexive digraphs to simplicial complexes: all details can be found in  \cite{MR2101222}, see also \cite{MR2232298}. A digraph $\bbP$ is a {\em poset} if its edge relation is reflexive, antisymmetric (if $(x,y)$ and $(y,x)$ are arcs then $x=y$) and transitive (if $(x,y)$ and $(y,z)$ are arcs so is $(x,z)$). We usually denote the relation on a poset by $\leq$. A poset $\bbP$ is a {\em chain} or {\em totally ordered set} if for every $x,y \in P$ either $x\leq y$ or $y \leq x$.

Let $\bbG$ be a digraph. Define a  simplicial complex $K(\bbG)$ as follows: a subset $S$ of $G$ is a simplex of $K(\bbG)$ if there exists a totally ordered set $\bbP$ and a homomorphism from $\bbP$ to $\bbG$ whose image is $S$.  We will call $K(\bbG)$ the {\em simplicial realisation of $\bbG$}, and we will say that $\bbG$ {\em triangulates} a topological space $X$ if  $X$ is homeomorphic to the geometric realisation of $K(\bbG).$

 Let $\bbG$ be a digraph. The {\em suspension of $\bbG$} is the digraph obtained from $\bbG$ by adding two vertices $u$ and $v$ that are adjacent to every vertex of $\bbG$ (by a two-way edge) and nothing else (see Figure \ref{figure-suspension}). It turns out that the geometric realisation of the suspension of $\bbG$ is homeomorphic to the suspension of the geometric realisation of $\bbG$ (hence the name).  
 
    \begin{figure}[htb]
\begin{center}
\includegraphics[scale=0.5]{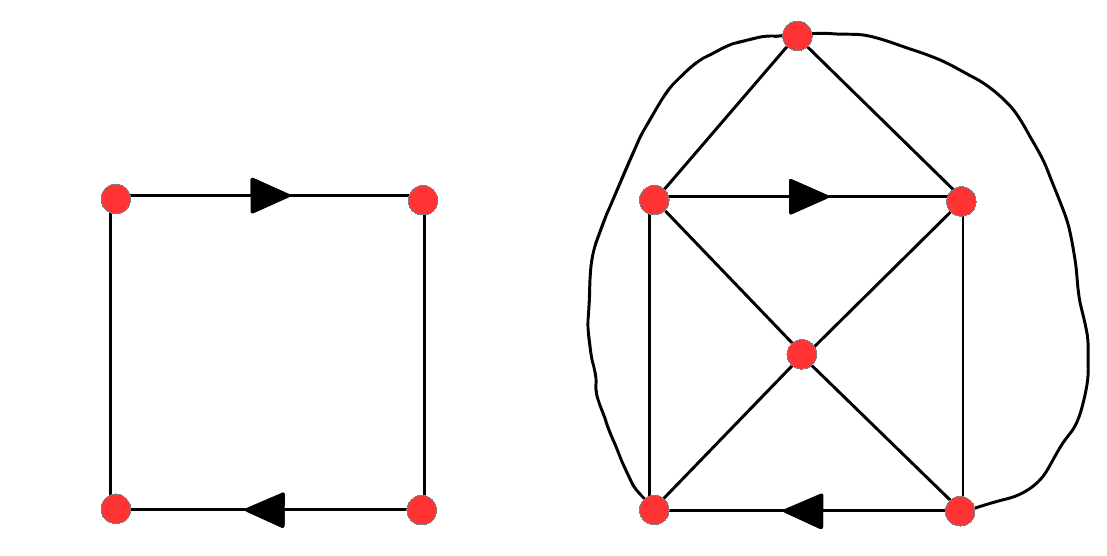}
 \caption{A digraph and its suspension. } \label{figure-suspension}
 \end{center}
\end{figure}

For posets we can refine the construction a bit. Given two posets $\bbP$ and $\bbQ$, their {\em ordinal sum} $\bbP \oplus \bbQ$ is the poset obtained from the disjoint union of $\bbP$ and $\bbQ$ by adding the relations $p \leq q$ for all $p \in P$ and $q \in Q$. An {\em antichain} is a totally disconnected poset, i.e. with no edges except loops. In ordinal sums, we denote the $n$-element antichain simply by $n$.  The {\em poset suspension of $\bbP$} is the poset $\bbP \oplus 2$. It turns out that the geometric realisation of the suspension of $\bbP$ is homeomorphic to the suspension of the geometric realisation of $\bbP$ (hence the name).  

%\bnote{put a figure to illustrate}

\section{A sufficient condition} \label{prelim-theorem} \label{sect-conditions}

In this section we present various necessary and sufficient conditions for a digraph to be S\l upecki. The main result, Theorem \ref{thm-embedding}, will be used in sections \ref{sect-gadgets} and \ref{sect-sums} and is a central tool in \cite{LLP}. 
 
 Let $\bbG$, $\bbH$ be  digraphs. The digraph $Hom(\bbG,\bbH)$ is defined as follows: its vertices are the homomorphisms  $\bbG \rightarrow \bbH$, and we have an arc $(f,g)$ if $(f(x),g(y))$ is an arc of $\bbH$ whenever $(x,y)$ is an arc of $\bbG$. It is easy to verify that  if $f,g,h \in Hom(\bbG,\bbG)$ and $f \rightarrow g$ then $h \circ f \rightarrow h \circ g$ (see Lemma 2.1 of \cite{MR2232298}).

\begin{lemma} \label{lemma-isolated-unary} Let $\bbG$  be a digraph such that the identity is an isolated loop in  $Hom(\bbG,\bbG)$. If $f:\bbG^p \rightarrow \bbG$ is an onto, essentially unary polymorphism then $f$ is an isolated loop in $Hom(\bbG^p,\bbG)$. \end{lemma}

\begin{proof} It suffices to show the result for projections: indeed,  if $f$ is an onto essentially unary polymorphism of arity $p$ then $f = \sigma \circ \pi$ for some projection $\pi$ and some automorphism $\sigma$ of $\bbG$. If $f \rightarrow g$ then $\pi \rightarrow \sigma^{-1} \circ g$ so if the result holds for projections then  $ \pi = \sigma^{-1} \circ g$ and $f = g$; similarly if $g \rightarrow f$.

Suppose that $\pi \rightarrow f$ in $Hom(\bbG^p,\bbG)$ where $\pi$ is a projection (the case  $f \rightarrow  \pi$ is identical). Without loss of generality, suppose that $\pi$ is the first projection. 
Fix $a_2,\dots,a_p \in G$ and define $g(x) = f(x,a_2,\dots,a_p)$. Then it is clear that $id \rightarrow g$ in $Hom(\bbG,\bbG)$.  Thus $g = id$ and hence $f$  is the first projection.  

\end{proof}

For $i=1,\dots,p$ let $\pi_i:\bbG^p \rightarrow \bbG$ denote the $i$-th projection.

\begin{lemma} \label{lemma-embedding} Let $\bbG$ be a digraph and let $f:\bbG^p \rightarrow \bbG$ be an onto polymorphism of arity at least 2. Then $(2) \Rightarrow (1)$:  
\begin{enumerate}
\item  there exists an embedding $e:\bbG \hookrightarrow \bbG^p$ such that the restriction of $f$ to $e(\bbG)$ is onto; 
\item  $f$ is essentially unary.
\end{enumerate}
if $\bbG$ is  idempotent trivial  such that the identity is an isolated loop in  $Hom(\bbG,\bbG)$, the converse holds. 
\end{lemma}

\begin{proof} $(2) \Rightarrow (1)$: suppose $f$ is onto and depends only on its $i$-th variable; then the required embedding is $e:\bbG \hookrightarrow \bbG^p$ defined by $e(x) = (a,a,\dots,a,x,a,\dots,a)$ where $a$ is any fixed vertex in $\bbG$ and $x$ appears in the $i$-th position. \\

Now assume that $\bbG$ is  idempotent trivial and that the identity is an isolated loop in  $Hom(\bbG,\bbG)$: we prove
$(1) \Rightarrow (2)$. First notice that if $\bbG$ is idempotent trivial then it must be connected: one can easily construct a binary idempotent polymorphism which is a non-projection if $\bbG$ is not connected. 
Let $f:\bbG^p \rightarrow \bbG$ and  $e:\bbG \hookrightarrow \bbG^p$ be such that $f(e(\bbG)) = \bbG$.  It follows that $f \circ e = \sigma$ for some automorphism $\sigma$ of $\bbG$. Define an operation $\phi:\bbG^p \rightarrow \bbG$ by 
 $$\phi(x_1,\dots,x_p) =   \sigma^{-1}f(\pi_1(e(x_1)),\dots, \pi_p(e(x_p))).$$
 Then $\phi(x,\dots,x) = \sigma^{-1}f(\pi_1(e(x)),\dots, \pi_p(e(x))) = \sigma^{-1}f(e(x)) = x$ and thus $\phi$ is idempotent. Since $\bbG$ is idempotent trivial  there exists some $i$ such that $\phi(x_1,\dots,x_p)=x_i$ for all $x_1,\dots,x_p$. We may suppose without loss of generality that $i=1$, and thus 
  $$f(\pi_1(e(x_1)),\dots, \pi_p(e(x_p))) =   \sigma(x_1) $$ for all $x_i$; in particular, 
 if we let $\bbH_i = \pi_i(e(\bbG))$ for all $i = 1,\dots,p$,  we conclude that $\bbH_1 = \bbG$. Let $k$ be the largest index $i$ such that $\bbH_i = \bbG$: notice that for each $1 \leq i \leq k$, $\sigma_i = \pi_i \circ e$ is an automorphism of $\bbG$. If $k = p$ then we've shown that $f$ depends only on its first variable. Otherwise, consider  the map $\psi: \bbG^{p-k} \rightarrow Hom(\bbG^k,\bbG)$ where $\psi(b_{k+1},\dots,b_p)$ is the homomorphism $(x_1,\dots,x_k) \mapsto f(x_1,\dots,x_k,b_{k+1},\dots,b_p)$. Clearly $\psi$ is a homomorphism, and since $\bbG^{p-k}$ is connected, so is the image of $\psi$. Choose elements $c_j \in \bbH_j$, $k+1 \leq j \leq p$. Then by the above identity we have that $\psi(c_{k+1},\dots,c_p)(x_1,\dots,x_k) = \sigma (\sigma_1^{-1}(x_1))$; in particular $\tau=\psi(c_{k+1},\dots,c_p)$ is an essentially unary onto polymorphism. 
 By Lemma \ref{lemma-isolated-unary}, it is an isolated loop and hence $\psi$ is a constant map with value $\tau$; in other words, $f(x_1,\dots,x_p) = \sigma (\sigma_1^{-1}(x_1))$ for all $x_j$ and we are done.
\end{proof}

We remark in passing that if the condition that $\bbG$ is idempotent trivial is dropped then the implication $(1) \Rightarrow (2)$ in the previous result is not necessarily true (the construction of a binary non-projection idempotent polymorphism on a disconnected digraph easily shows this.)

%\begin{corollary} \end{corollary}

\begin{theorem}  \label{thm-embedding} Let $n \geq 1$ and let $\bbG$ triangulate an $n$-sphere. If for every $p \geq 2$ and every onto polymorphism $f:\bbG^p \rightarrow \bbG$  there exists an embedding $e:\bbG \hookrightarrow \bbG^p$ such that the restriction of $f$ to $e(\bbG)$ is onto then  $\bbG$ is S\l upecki.
  \end{theorem}
  
  \begin{proof} If $\bbG$ triangulates an $n$-sphere then  the $n$-th homotopy group of the geometric realisation of $\bbG$ is non-trivial, but every proper retract of $\bbG$ has a contractible realisation and hence has trivial homotopy. It follows from Theorem 2.11 of \cite{MR2232298} and Claim 1 in the proof of that same theorem that $\bbG$ is idempotent trivial and that the identity is alone in its connected component of $Hom(\bbG,\bbG)$. We can then invoke Lemma \ref{lemma-embedding} to conclude.  \end{proof}

For completeness' sake, we state an analog of Lemma \ref{lemma-embedding}  when we only consider strong components. Write $f \sim g$ to mean that $f$ and $g$ are in the same strong component of $Hom(\bbG,\bbG)$. 

\begin{lemma} \label{lemma-isolated-unary-strong} Let $\bbG$  be a digraph such that the identity is { alone in its strong component of}  $Hom(\bbG,\bbG)$. If $f:\bbG^n \rightarrow \bbG$ is an onto, essentially unary polymorphism then $f$ is { alone in its strong component of} $Hom(\bbG^n,\bbG)$. \end{lemma}

\begin{proof}  It suffices to show the result for projections: indeed,  if $f$ is an onto essentially unary polymorphism of arity $n$ then $f = \sigma \circ \pi$ for some projection $\pi$ and some automorphism $\sigma$ of $\bbG$. If $f \sim g$ then $\pi \sim \sigma^{-1} \circ g$ so if the result holds for projections then  $ \pi = \sigma^{-1} \circ g$ and $f = g$.

Suppose that $\pi \sim  f$ in $Hom(\bbG^n,\bbG)$ where $\pi$ is a projection. Without loss of generality, suppose that $\pi$ is the first projection. 
Fix $a_2,\dots,a_n \in G$ and define $g(x) = f(x,a_2,\dots,a_n)$. Then it is easy to see  that $id \sim g$ in $Hom(\bbG,\bbG)$.  Thus $g = id$ and hence $f$  is the first projection.  

\end{proof}

Now we can state the analog of Lemma \ref{lemma-embedding}  for strongly connected digraphs: in this case we can weaken the technical condition on the identity.

\begin{lemma} \label{lemma-embedding-strong} Let $\bbG$ be a strongly connected digraph and let $f:\bbG^n \rightarrow \bbG$ be an onto polymorphism of arity at least 2.  If $\bbG$ is  idempotent trivial  such that the identity is alone in its strong component of  $Hom(\bbG,\bbG)$, then (1) implies (2):
\begin{enumerate}
\item  there exists an embedding $e:\bbG \hookrightarrow \bbG^n$ such that the restriction of $f$ to $e(\bbG)$ is onto; 
\item  $f$ is essentially unary.
\end{enumerate}

\end{lemma}

\begin{proof} The proof is exactly the same as for Lemma \ref{lemma-embedding}, except for two observations: (i) we replace the use of Lemma \ref{lemma-isolated-unary} by that of Lemma \ref{lemma-isolated-unary-strong} and (ii) since $\bbG$ is strongly connected, so is the image of the homomorphism $\psi$; since its image contains a surjective essentially unary polymorphism, this image must be constant by (i). \end{proof}

We close this section with some observations on the conditions used in the previous lemmas. The technical condition that the identity should be an isolated loop is slightly vexing; it can be removed in many instances, but unfortunately not always. 

We say the  digraph $\bbG$  is {\em intransitive} if its simplices all have dimension at most 1, i.e. there are no transitive triples in $\bbG$.

 \begin{lemma}\label{629952} Let $\bbG$ be a poset or symmetric digraph or intransitive digraph. If $\bbG$ is idempotent trivial then the identity is isolated in $Hom(\bbG,\bbG)$.   \end{lemma}
 
 As we argued earlier, we can assume that $\bbG$ is connected if it is idempotent trivial. 

\begin{proof}  (1) Suppose $\bbG$ is a poset, and $id$ is not alone in its connected component. Then (see the comment following Lemma 2.10 in \cite{DBLP:journals/dm/MarotiZ12}) without loss of generality $id < f$ where $f$ differs from $id$ only in one place, say $f(a)=b$; it is easy to see that in this case $b$ is a unique upper cover of $a$ in $\bbG$ (i.e. if $a<c$ then $b \leq c$); then define $\phi(x,y) = y$ if $x \leq a$ and $\phi(x,y)= f(y)$ otherwise. It is easy to see that $\phi$ is an idempotent polymorphism, and that it is not a projection. %(idea: to make sure $f$ is idempotent, it suffices not to map $a$ to $f$ in the map from $\bbG$ to $Hom(\bbG,\bbG)$.)   

(2) Suppose that $\bbG$ is a symmetric digraph, and that $id$ is not alone in its component, so without loss of generality let $id \rightarrow f$. By Lemma 2.2 of \cite{DBLP:journals/dm/MarotiZ12}, we can assume that $f$ has at least one fixed point. Then define $\phi(x,y)= f(y)$ if $x=a$ and $\phi(x,y)=y$ otherwise. It is easy to see that $\phi$ is an idempotent polymorphism and that it is not a projection. 
%(Idea: the pair $\{id,f\}$ is a clique so we can map $\bbG$ to this clique in any way we wish; but to make sure $\phi$ is idempotent  we must make sure that only fixed %points of $f$ are mapped to it. )  

(3) Suppose that $\bbG$ is an intransitive digraph.  Suppose that the identity is not alone in its component, so without loss of generality let $f\rightarrow id$. Let $x,y \in G$ be distinct such that $x\rightarrow y$. Then $\{f(x),x,y\}$ is a simplex, and hence $f(x) \in \{x,y\}$. Similarly, $\{f(x),f(y),y\}$ is a simplex. We want to show that $(f(x),f(y)) \in \{(x,y),(y,x)\}$, and that $(f(x),f(y)) = (x,y)$ in the case the edge is non-symmetric.

\noindent{Case 1.} Suppose that $\{x,y\}$ is not a symmetric edge. Then $f(x)=x$ and $f(y) \in \{x,y\}$. If $f(y)=y$ we are done; otherwise $f(y)=x$; we show that $y$ has a unique in-neighbour (namely $x$) and no out-neighbour. First suppose $y\rightarrow w$. Then $x=f(y)\rightarrow w$ means that $\{x,y,w\}$ is transitive so $y=w$. Now suppose that $z \rightarrow y$. We now know this edge in non-symmetric, so by our first observation $f(z) = z$. Then $z=f(z)\rightarrow f(y)=x$ implies that $\{z,x,y\}$ is a transitive triple, and hence $z \in \{x,y\}$. 

Define a map $\phi:G^2\rightarrow G$ by  

$$
\phi(u,v)=
\begin{cases}
x, & \text{if }v=y \text{ and } u \neq y,  \\
v, & \text{otherwise.}
\end{cases}
$$

It is straightforward to verify that $\phi$ is an idempotent polymorphism of $\bbG$ and is not a projection, contradicting our hypothesis. Hence $f(x)=x$ and $f(y)=y$ if the edge is non-symmetric. 

\noindent{Case 2.} Now assume the edge is symmetric and that $f(x)=x$. One can argue as above that if $f(y)\neq y$ then $y$ is a pendant vertex (with symmetric edge), and mimicking Case 1, we can define a binary idempotent polymorphism which is not a projection. We conclude that $f(y)=y$. 
%using the retraction of $\bbG$ onto $\bbG \setminus \{y\}$, which is a symmetric neighbour of $id$. 
Now if $f(x)=y$, by our previous argument (exchanging the roles of $x$ and $y$) we must have $f(y)\neq y$ and hence $f(y)=x$. 

So we conclude that for every non-symmetric edge $x\rightarrow y$ we have $f(x)=x$ and $f(y)=y$, and for a symmetric edge $(x,y)$ we have $f(x)=x$ and $f(y)=y$ or $f(x)=y$ and $f(y)=x$. However, suppose that $(x,y)$ is such a flipped edge. Then it is easy to see that neither $x$ nor $y$ can have any other neighbour (symmetric or not), so $\bbG$ is an edge. Since the edge admits all polymorphisms it is not idempotent trivial. Consequently, $f$ fixes all vertices that are incident to some arc, and hence the identity is alone in its connected component.

 \end{proof}

A {\em pre-order} is a reflexive, antisymmetric relation. If $A$ and $B$ are strong components of the digraph $\bbG$, define $A \sqsubseteq B$ if there exist $a \in A$ and $b \in B$ such that $a \rightarrow b$. It is immediate that this  defines a pre-order on the set of strong components of $\bbG$.  An element $a$ of a pre-order $\bbK$ is {\em minimal} if $b \rightarrow a$ implies $a=b$ for all $b \in K$; maximal elements are defined dually. 

\begin{lemma} \label{lemma-not-strong-id-alone}Let $\bbG$ be a S\l upecki digraph which is not strongly connected. Then the identity is alone in its connected component of $Hom(\bbG,\bbG)$. \end{lemma}

\begin{proof} Suppose that $id \rightarrow r$ where $r \neq id$ (the case $r \rightarrow id$ is identical.)
Since $\bbG$ is not strongly connected, its strong components form a pre-order. Choose a strong component $A$ which is minimal in the pre-order, and define  a binary polymorphism $f$ as follows: $f(x,y) = y$ if $x \in A$ and $f(x,y)=r(y)$ otherwise. It is obviously onto, and it is easy to verify that it is a polymorphism. Let $a \in A$, $b \not\in A$ and let $z \neq r(z)$. Then $f(a,z) = z$ and $f(b,z)=r(z)$ so $f$ depends on its first variable. And since $f(a,y) = y $ for all $y$ $f$ depends on its second variable. 
\end{proof}

\begin{comment}
{\bf Remark.} In the previous proof, the operation that we construct is idempotent if the image $R$ of $r$ is disjoint from $A$ (notice that by Lemma 2.2 of M-Z we can assume that $r$ is a retraction). Can we always choose $A$ in that way ? This would answer the second question below. {\bf answer}: No, not for a fixed $r$. Consider the digraph obtained from Lemma \ref{lemma-example} below, adding a sink $w$. The retraction $r$ onto $\{1,2,3,w\}$ satisfies $id \rightarrow r$, but there is only one minimal $A$ (namely, the digraph without the sink) and $r$ intersects it. However, we also have $id \rightarrow f$ where $f$ is the constant map to $w$, and then we can define the required idempotent. 
\end{comment}

One may ask whether the conditions above can be relaxed. The following example shows that the S\l upecki condition does {\em not} imply the identity is alone in its connected component.

\begin{lemma} \label{lemma-example} There exists a reflexive digraph $\bbG$ with the following properties:
\begin{itemize}
\item $\bbG$ is S\l upecki;
\item $\bbG$ is strongly connected; 
\item there are arcs $s \rightarrow id \rightarrow r$ in $Hom(\bbG,\bbG)$ such that $r$ and $s$ are retractions  onto the same subdigraph of size $|G|-1$;
\item the connected component of the identity is exactly $\{id,r,s\}$, in particular, the identity is alone in its strong component. 
\end{itemize} 
\end{lemma}

   \begin{figure}[htb]
\begin{center}
\includegraphics[scale=0.5]{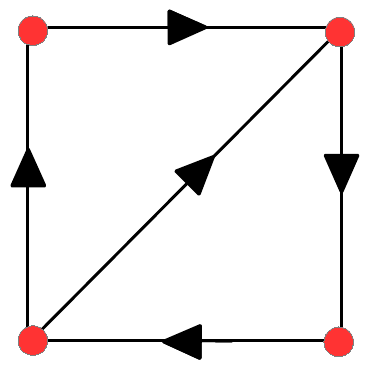}
 \caption{A S\l upecki digraph where the identity is not isolated. }
 \end{center}
\end{figure}

\begin{proof} Let $\bbG$ be the reflexive digraph with vertices $\{0,1,2,3\}$ and arcs (other than loops) $\{(0,1),(1,2),(2,3),(3,0),(3,1)\}$. It is easy to verify that there are two retractions of $\bbG$ onto the subdigraph $\bbH$ induced by $H=\{1,2,3\}$, one sending 0 to 1 (call it $r$) the other sending 0 to 3 (call it $s$); we have that $s \rightarrow id \rightarrow r$, and no other arcs between these homomorphisms. 

Now we show that $\bbG$ is S\l upecki; let $f:\bbG^n \rightarrow \bbG$ be an $n$-ary, surjective polymorphism, where $n \geq 2$. Let $\Delta$ denote the diagonal of $\bbG^n$.  It is immediate that the only strongly connected subdigraphs of $\bbG$ are singletons, $\bbH$ and $\bbG$. Notice that $\bbH$ is a directed cycle and hence  is S\l upecki by Proposition \ref{prop-directed-cycles}.  \\

\noindent{\bf Claim 1.} $f(\Delta)$ contains $H$.   

Indeed, since $\Delta$ is strongly connected, so is its image under $f$. By the last remark, it thus suffices to prove that $f$ is not constant on $\Delta$; suppose for a contradiction that $f(\Delta) = \{a\}$. If $u \rightarrow v$ in $\bbG$, then $(u,\dots,u) \rightarrow \overline{x} \rightarrow (v,\dots,v)$ for all $\overline{x} \in \{u,v\}^n$;   applying $f$, and since $\bbG$ contains no symmetric edges, we conclude that $f(\overline{x})=a$. If $\overline{x} \in \{0,2\}^n$, notice that $\overline{a} \rightarrow \overline{x} \rightarrow \overline{b}$ where $\overline{a}$ is obtained from $\overline{x}$ by replacing all 2's by 1's and $\overline{b}$ is obtained from $\overline{x}$ by replacing all 0's by 1's; applying $f$ we obtain again that $f(\overline{x})=a$. Hence $f$ has value $a$ on any tuple with at most two distinct coordinates. Now if $u \rightarrow v \rightarrow w$ in $\bbG$, we have that $\overline{a} \rightarrow \overline{x} \rightarrow \overline{b}$ where $\overline{a}$ is obtained from $\overline{x}$ by replacing all v's by u's and $\overline{b}$ is obtained from $\overline{x}$  by replacing all v's by w's, and this for all $\overline{x} \in \{u,v,w\}^n$. Since every 3-element subset of $G$ has this form, we conclude, as before, that $f$ has value $a$ on any tuple with at most three distinct coordinates. It is easy to see that an analogous argument will show that $f$ must also have value $a$ on all tuples with 4 distinct coordinates,  a contradiction since $f$ is onto. \\

\noindent{\bf Claim 2.} Let $\Delta' = \Delta \cap H^n$. Then $f(\Delta')=H$.  

By Claim 1, we certainly have that $|f(\Delta')| \geq 2$ and so it is not a singleton; since $\Delta'$ is strongly connected, so is $f(\Delta')$ and thus it must contain $H$ and we are done.\\

\noindent{\bf Claim 3.} $f(H^n)=H$.

Otherwise, by Claim 2, we have some $\overline{x} \in H^n$ such that $f(\overline{x})=0$. Let $\overline{y} \in H^n$ such that $f(\overline{y})=1$ (it exists by Claim 2). Now clearly we have the following: for any $u,w \in \bbH$ there exists $v \in \bbH$ such that $u \rightarrow v \rightarrow w$, and hence the same property holds in $\bbH^n$.   Then there exists $v$ such that $\overline{y} \rightarrow v \rightarrow \overline{x}$; but then $1\rightarrow f(v)\rightarrow 0$, a contradiction. \\

By Claim 3, and since $\bbH$ is S\l upecki, we conclude that there exists an automorphism $g:\bbH \rightarrow \bbH$, such that, without loss of generality, $f(x_1,\dots,x_n) = g(x_1)$ for  all $x_i \in H$. \\

\noindent{\bf Claim 4.} Let $\overline{x}$ be a tuple such that $x_1 \neq 0$. Then $f(\overline{x}) = g(x_1)$. 

We have $\overline{a} \rightarrow \overline{x} \rightarrow \overline{b}$ where $\overline{a}$ is obtained from $\overline{x}$ by replacing all 0's by 3's and $\overline{b}$ is obtained from $\overline{x}$ by replacing all 0's by 1's; applying $f$ we obtain that $f(\overline{x}) = f(\overline{a})=f(\overline{b}) = g(x_1)$ (since $\bbG$ has no symmetric edges). \\

Now take {\em any} $\overline{x} \in \{0\} \times \bbG^{n-1}$; then $\overline{a} \rightarrow \overline{x} \rightarrow \overline{b}$ where $\overline{a}$ is obtained from $\overline{x}$ by replacing the first coordinate by 3 and $\overline{b}$ is obtained from $\overline{x}$ by replacing the first coordinate by 1; applying $f$ we obtain using Claim 4 that $g(3) \rightarrow f(\overline{x}) \rightarrow g(1)$. It follows that there exists a value $a \in \bbH$ such that $f(\overline{x}) = a$ for all $\overline{x} \in \{0\} \times \bbG^{n-1}$ (this value $a$ depends on the automorphism $g$). In particular, the image under $f$ of the strongly connected digraph $\{0\} \times \bbG^{n-1}$ must be a singleton, and thus $f$ depends only on its first variable. \\

Finally we show that the component of the identity is $\{id,r,s\}$. Consider the maps from $Hom(\bbG,\bbG)$ to $Hom(\bbH,\bbH)$ sending $f$ to $(r \circ f)|_\bbH$ and $(s \circ f)|_\bbH$ respectively. These are clearly homomorphisms, and thus must map the entire component of $id_\bbG$ to the component of $id_\bbH$ which   by Lemma \ref{lemma-not-strong-id-alone} is $\{id_\bbH\}$. A simple verification shows that this forces any member $f$ of the component of $id_\bbG$ to be the identity when restricted to $\bbH$, which forces $f \in \{id,r,s\}$.

\end{proof}

 \section{Some S\l upecki Spheres} \label{sect-gadgets}
 
 As we remarked earlier, digraphs that triangulate spheres are known to be idempotent trivial, and cycles of girth at least 4 are in fact S\l upecki. In the next section we will provide examples of digraphs that triangulate 2-spheres that are {\em not} S\l upecki; in the present section we provide examples of digraphs triangulating spheres that are S\l upecki, using a technique involving gadgets. We also state for the record the result that digraphs that triangulate 1-spheres are S\l upecki.

An $n$-ary operation $f$ on a set $A$ {\em preserves} the $k$-ary relation $\theta \subseteq A^k$ if applying $f$ to the rows of any $k \times n$ matrix whose columns are tuples in $\theta$ yields a tuple in $\theta$. 
 
 \begin{lemma}[\cite{MR2254622}]\label{lemma-slupecki} Let $A$ be an $n$-element set, $n \geq 2$. Then the operations on $A$ preserving the relation
 $$\theta = \{(a_1,\dots,a_n): |\{a_1,\dots,a_n\}| < n \}$$
 are precisely the essentially unary operations and the non-surjective operations on $A$.  \end{lemma}
 
 The set of operations preserving the above relation $\theta$ is called the {\em S\l upecki} clone; it is known to be a maximal clone, and it is easy to see that it contains all non-surjective operations and all unary operations. To show that a reflexive digraph is S\l upecki, it is necessary and sufficient to prove that its polymorphisms preserve the relation $\theta$; it is well-known (see \cite{MR2254622}) that this is equivalent to showing that there exists a primitive, positive definition of $\theta$ in terms of the edge relation of the digraph. Or in more convenient terms for our purposes:
 
 \begin{lemma} The digraph $\bbG$ is S\l upecki if and only if there exists a digraph $\bbK$ and vertices $x_1,\dots,x_n$ of $\bbK$ such that 
   $$\theta = \{(f(x_1),\dots,f(x_n)): f:\bbK \rightarrow \bbG\}.$$ \end{lemma}
 
 If the above holds we say that the gadget $\bbK$ {\em pp-defines} the S\l upecki relation $\theta$. It turns out that finding the gadget $\bbK$ is fairly easy for various families of cycles (see below), but we do not know of a uniform gadget construction for all cycles. 
 %these were first shown to be S\l upecki by Demetrovics and Ronyai \cite{MR1020459}.
 %We can also do some other examples of cycles (see below). Unfortunately, the trick we use provably cannot work for odd symmetric cycles, i.e. the structure $\bbK$ of %Theorem \ref{theorem-uniform-gadget} does not exist for the 5-cycle (and probably neither for all odd cycles).  

 \begin{definition} Let $\bbG$ be a digraph, and let $S \subseteq G$. If there exists a digraph $\bbK$, vertices $x_1,\dots,x_k,u \in K$ and a partial map $f:\{x_1,\dots,x_k\} \rightarrow G$ such that \\ $S = \{g(u): g:\bbK \rightarrow \bbG, \, g(x_i)=f(x_i) \, \text{ for all } i \}$ then we say {\em $S$ is pp-defined by $(\bbK,f,u)$.}\end{definition}

%Notice that every subset of $\bbG$ can be pp-defined by {\em some} triple $(\bbK,f,u)$ precisely when $\bbG$ is quasiprojective\footnote{A structure is {\em %quasiprojective} if every of its polymorphism is conservative, i.e. satisfies $f(x_1,\dots,x_n) \in \{x_1,\dots,x_n\}$ for all $x_i$; it is known that for posets and symmetric 
%(reflexive ? check !) graphs, $\bbG$ is quasiprojective iff it is idempotent trivial. Question: is this true for reflexive digraphs in general ?} ; in particular, if $\bbG$ is %idempotent-trivial, then all subsets are pp-defined this way.

\begin{theorem} \label{theorem-uniform-gadget} Let $\bbG$ be a digraph. If there exists a digraph $\bbK$ with vertices $x_1,\dots,x_k,u$ such that the conditions below hold, then $\bbG$ is S\l upecki.
\begin{enumerate}
\item for all $a \in G$, there exists some $f_a$ such that $(\bbK,f_a,u)$ pp-defines $G \setminus \{a\}$;
\item for all $f$, the set pp-defined by $(\bbK,f,u)$ is a proper subset of $G$. 
\end{enumerate}
 \end{theorem}
 
   \begin{figure}[htb]
\begin{center}
\includegraphics[scale=0.5]{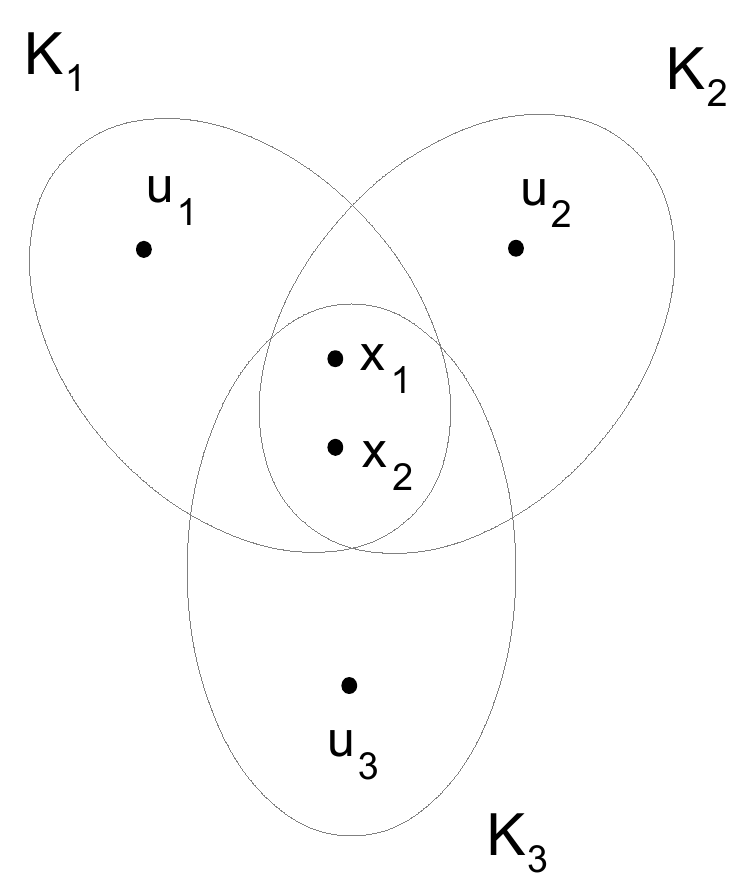}
 \caption{A gadget $\bbL$ with $n=3$ and $k=2$. }
 \end{center}
\end{figure}

\begin{proof} Let $n = |G|$. Glue $n$ distinct copies of $\bbK$ by identifying their vertices $x_1,\dots,x_k$; rename their respective vertices $u$ as $u_1,\dots,u_n$. We claim that this structure $\bbL$ pp-defines the S\l upecki relation $\theta$. Indeed, let $g:\bbL \rightarrow \bbG$; by the second property, we know that the values of $g$ on $u_1,\dots,u_n$ can never be all distinct; and when  the  restriction of $g$ to $x_1,\dots,x_n$ is equal to $f_a$, then we will obtain all $n$-tuples that miss the value $a$.  \end{proof}

 Notice that if $k=1$, condition (2) of Theorem \ref{theorem-uniform-gadget} is implied by condition (1). \\

   %This can be seen as we need graphs that do not contain a homomorphic image of a $2-chain$, where no vertex has degree $1$ and the Euler characteristic of the circle %is 0 (which coincides with the alternating sum $k_0 - k_1$ where $k_i$ is the number of simplices of dimension $i$ in $K(\bbG).$ \bnote{fix this}

 \begin{prop} \label{prop-directed-cycles}[\cite{MR2480633}] A directed cycle of girth at least 3 is S\l upecki. \end{prop}
 
 \begin{proof} We invoke Theorem \ref{theorem-uniform-gadget}. Let $m$ be the girth of $\bbG$. The gadget $\bbK$ consists of a directed path of length $m-2$, with  starting point $x_1 = x$ and end point $u$. Notice that the only vertex not reachable from $i+1$ in $\bbG$ by a directed path of length $m-2$ is $i$, hence $(\bbK,f,u)$ pp-defines $G \setminus \{i\}$ when $f(x) = i+1$.   
 \end{proof}
 
 \begin{prop} A symmetric even cycle of girth at least 4 is S\l upecki. \end{prop}
 
 \begin{proof} Let $\bbG$ be the cycle of girth $n=2m$; we invoke Theorem \ref{theorem-uniform-gadget}: the gadget $\bbK$ consists of a path of length $m-1$ with  starting point $x_1=x$ and endpoint $u$. Notice that the only vertex not reachable from $a$ in $\bbG$ by a path of length $m-1$ is its antipode $b$; hence  
 $(\bbK,f,u)$ pp-defines $G \setminus \{b\}$ when $f(x) = a$.
 \end{proof}
 
 \begin{definition} Let $m \geq 2$. The {\em $2m$-crown} is the reflexive digraph on $\{0,1,\dots,2m-1\}$ with arcs $(2i,2i \pm 1)$, $0 \leq i \leq m-1$ (modulo $2m$).   \end{definition}
 
   \begin{figure}[htb]
\begin{center}
\includegraphics[scale=0.5]{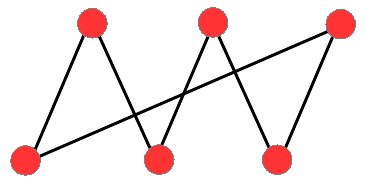}
 \caption{The 6-crown. Arcs are all oriented bottom to top. }
 \end{center}
\end{figure}

% In other words, its edge structure is given by the word $(+-)^m$. These are well-known, well-studied posets. They were first shown to be S\l upecki by Demetrovics and %Ronyai\cite{MR1020459}.  
 \begin{prop}[\cite{MR1020459}] Every crown is S\l upecki. \end{prop}
 
   \begin{figure}[htb]
\begin{center}
\includegraphics[scale=0.4]{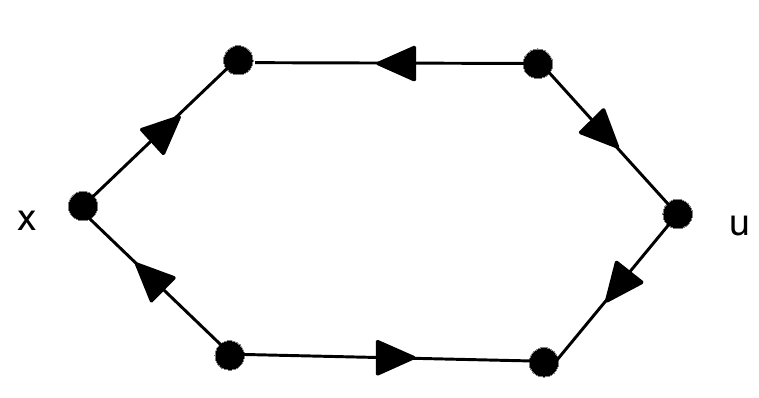}
 \caption{The gadget for the 6-crown.}
 \end{center}
\end{figure}
 
 \begin{proof} We invoke Theorem \ref{theorem-uniform-gadget}: the gadget $\bbK$ consists of two paths $\bbP$ and $\bbQ$ glued at their start point (call this $x_1=x$) and also glued at their end point, call that one $u$: $\bbP$ is an alternating sequence of forward and backward edges of length $m$, and $\bbQ$ is an alternating sequence of the same length but opposite orientations. We claim that if we set $f(x) = a$, then $(\bbK,f,u)$ pp-defines $G \setminus \{a+m\}$. Indeed, by symmetry of $\bbK$ and $\bbG$ we may assume without loss of generality that $a=0$; then using the path $\bbQ$, we can reach all vertices of $\bbG$ except $m$ (because in the first step the backward edge forces the value to remain 0, and this leaves a path of length $m-1$ so we cannot reach $m$). Notice that for $a=0$, using the path $\bbP$ we can reach all vertices. This proves our claim. \end{proof}

Here is a rather ad hoc example. 

 \begin{prop} Let $\bbG$ be the 4-cycle with the following edge structure:\\
  $E(\bbG)= \{(0,1), (1,2),(2,1), (2,3),(3,0),(0,3)\}$. Then $\bbG$ is S\l upecki. \end{prop}
  
  \begin{figure}[htb]
\begin{center}
\includegraphics[scale=0.5]{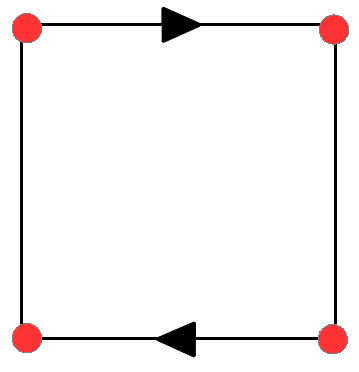}
 \caption{An ad hoc example of a 4-cycle. }
 \end{center}
\end{figure}

\begin{proof}  Let $\bbK$ be a directed 4-cycle; choose some vertex $x_1=x$ and let $u$ be its antipode. Suppose that $f(x)=1$; then it is easy to see that $f(u) \in \{1,2,3\}$. By symmetry of $\bbG$, if $f(x) = 3$ then $(\bbK,f,u)$ pp-defines $\{0,1,3\}$. Finally, if we reverse all arcs, the same argument shows that $(\bbK,f,u)$ pp-defines the other two 3-element subsets for some choice of $f$. We can now apply Theorem \ref{theorem-uniform-gadget}.

 \end{proof}

 \begin{lemma} \label{lemma-spheres}A digraph $\bbG$ triangulates a 1-sphere if and only if $\bbG$ is a directed 3-cycle or a cycle of girth at least 4. \end{lemma}

\begin{proof}
Let $\bbG$ triangulate a 1-sphere. Then $\bbG_{Sym}$ is connected,   has no pendant vertex, and it has Euler characteristic 0, i.e. the number of vertices is equal to the number of edges, and thus every vertex has degree exactly 2. Hence $\bbG$ is a cycle. Furthermore, there cannot be any simplex of size 3, thus if it is a 3-cycle then it is the directed 3-cycle. Consequently, the digraphs that triangulate a 1-sphere are precisely the directed 3-cycle and the cycles of girth at least 4. \end{proof}

\begin{theorem} \label{theorem-big-paper} If $\bbG$ triangulates a 1-sphere, then $\bbG$ is S\l upecki.
\end{theorem}

\begin{proof} This follows from the previous lemma, Proposition \ref{prop-directed-cycles} and the main result of \cite{LLP}. \end{proof}

\begin{prop} Let $n \geq 2$. Let $\bbG_n$ be obtained from the complete graph on $2n$ vertices by removing the edges $\{i,i+n\}$ for all $0\leq i \leq n-1$. Then
\begin{enumerate}
\item $\bbG_n$ triangulates $\bbS^{n-1}$;
\item $\bbG_n$ is S\l upecki.
\end{enumerate}
\end{prop}

\begin{proof} (1) Notice that $\bbG_2$ is a 4-cycle and thus triangulates $\bbS^{1}$; and it is easy to see that $\bbG_n$ is the suspension of $\bbG_{n-1}$. 
(2) We invoke Theorem \ref{theorem-uniform-gadget}: we use as gadget $\bbK$ a symmetric edge with vertices $x$ and $u$. Once having chosen the image of $x$, the only vertex not in its reach is its antipode.  

\end{proof}

\begin{prop} Let $n \geq 3$. Let $\bbH_n$ be obtained from the complete graph on $n$ vertices by removing the directed cycle $(0,1),\dots,(n-1,0)$. Then
\begin{enumerate}
\item $\bbH_n$ triangulates $\bbS^{n-2}$;
\item $\bbH_n$ is S\l upecki.
\end{enumerate}
\end{prop}

%(Note: these digraphs are mentioned in \cite{MR2232298})
\begin{proof}  (1) It is easy to see that the simplicial complex of $\bbH_n$ is a truncated boolean lattice of height $n-2$. 
(2) To show that $\bbH_n$ is S\l upecki, we let the gadget $\bbK$  be a directed edge from $x$ to $u$. Again, once chosen the value for $f(x)=i$, by the construction of our graph, the only non-reachable vertex is $i+1$ (mod n). We conclude by using  Theorem \ref{theorem-uniform-gadget}.  

\end{proof}

\section{Ordinal sums}\label{sec:SumOrd} \label{sect-sums}

In this section we investigate the nature of surjective polymorphisms on posets of the form $m \oplus n $ and $m \oplus n \oplus k$ with $m,n,k \geq 2$. 
 %The natural number $s$ denotes the antichain with $s$ elements.  
These posets are known to be idempotent trivial  {\cite{DBLP:journals/ijac/Larose95}, and, hence, by Lemma \ref{629952}, the identity is an isolated loop in $Hom(\bbP,\bbP)$. 
These rather simple-looking posets will provide many examples of unexpected behaviour. 
Namely, although $m\oplus n$ is S\l upecki, none of the posets $n\oplus m\oplus k $ are; this includes  the poset $2 \oplus 2 \oplus 2$ which is the poset suspension of a 4-cycle and hence triangulates a 2-sphere. %Furthermore, some of them {\em are} actually 2-S\l upecki! 
When considering homomorphisms between posets we often use the term {\em monotone} instead.

 \begin{theorem} Let $m,n \geq 2$. Then $\bbP = m \oplus n$ is S\l upecki. \end{theorem}
 
 \begin{figure}[htb]
\begin{center}
\includegraphics[scale=0.5]{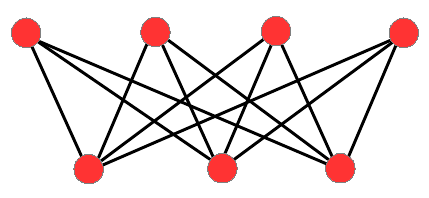}
 \caption{The poset $3 \oplus 4$. Arcs are all oriented bottom to top. } \label{gadget}
 \end{center}
\end{figure}

 \begin{proof} We are going to use Lemma \ref{lemma-embedding} in the $(1) \Rightarrow (2)$ direction. 
 We have just observed that the two technical conditions, namely that $\bbP$ is idempotent trivial and the identity is isolated, are fulfilled.
 Let $f:\bbP^s \rightarrow \bbP$ ($s\geq2$) be an onto polymorphism.
 Finding the embedding $e$ of Lemma \ref{lemma-embedding} means finding an isomorphic copy of $\bbP$ in $\bbP^s$ on which $f$ is onto.
 That is what we are going to do.
 
 Let $A$ and $B$ denote the set of minimal and maximal elements of $\bbP$, respectively.  
 We claim $f(A^s) \supseteq A$. 
 Pick $a\in A$, then $f(x)=a$ for some $x\in \bbP^s$, and, under $x$, there is a minimal element $y$ of  $\bbP^s$ for which $f(y)=a$, and, from minimality, we have $y \in A^s$. 
 %This is clear since $f$ is onto and every element of $\bbP^2$ is above some minimal element which is in $A^2$. 
 Dually, $f(B^s) \supseteq B$. 
 Now we can choose two sets of elements $X_A \subseteq A^s$ and $X_B \subseteq B^s$ such that 
 $$|X_A|=|A|, \; |X_B|=|B|, \; f(X_A)=A, \text{ and } f(X_B)=B.$$
 As the elements of $A^s$ and $B^s$ form two antichains such that the elements of $A^s$ are all under those of $B^s$, the poset $X_A \cup X_B$ is in fact an isomorphic copy of $\bbP$ in $\bbP^s$ on which $f$ is onto. 
 We are done. 
 %Then clearly we have an embedding $e$ of $\bbP$ into $\bbP^2$ such that $f(e(\bbP))=\bbP$. Since the poset is idempotent trivial and the identity is isolated, we conclude from Lemma \ref{lemma-embedding} that $f$ is essentially unary.  
 \end{proof}

The case of  ordinal sums of three antichains is more interesting.
First some notation. Once $m,n,k \geq 2$ are fixed, we let $A=\{a_0,\dots,a_{m-1}\}$,  $B=\{b_0,\dots,b_{n-1}\}$ and $C=\{c_0,\dots,c_{k-1}\}$ denote the minimal, middle and maximal elements of $\bbP=m \oplus n \oplus k$ respectively.

We first show that $m \oplus n \oplus k$ is not 3-S\l upecki (and hence not S\l upecki)  by directly exhibiting a not essentially-unary ternary polymorphism.

\begin{theorem} \label{thm-not-slu} Let  $m,n,k \geq 2$. Then $\bbP=m \oplus n \oplus k$ is not 3-S\l upecki, and, hence, not S\l upecki.  \end{theorem}
  
  \begin{proof}
We exhibit a polymorphism $f:\bbP^3 \to \bbP$ that is onto and not essentially unary: let

$$
f(x,y,z)=
\begin{cases}
s, & \text{if }(x,y,z)=(s,s,s), \;\;  \text{for }s \in A\cup C \\
b_i, & \text{if }(x,y,z)=(a_0,c_0,b_i), \;\; (i=0, \dots, n-1) \\
a_0, & \text{if }(x,y,z)<(a_0,c_0,b_i), \;\; (i=0, \dots, n-1) \\
c_0, & \text{if }(x,y,z)>(a_0,c_0,b_i), \;\; (i=0, \dots, n-1) \\
b_0, & \text{otherwise.}
\end{cases}
$$

\noindent First, we have to verify that our function is well-defined.
There are, indeed, elements of $P^3$ that are given values multiple times in the definition;  we need to make sure  that they are  given the same value.
It is clear that this occurs for $(a_0,a_0,a_0)$ and $(c_0,c_0,c_0)$ which are given the values $a_0$ and $c_0$ respectively, at each of their occurrences in the definition, hence there is no conflict here.
Let $x^{\uparrow}=\{y: x\leq y\}$ and similarly $x_{\downarrow}=\{y: x\geq y\}$.
The only other way this multiple assignment phenomenon can happen for $(x,y,z) \in P^3$ is if it belongs to 
$(a_0,c_0,b_i)^{\uparrow} \cap (a_0,c_0,b_j)^{\uparrow}$ 
(or its dual) for $i\neq j$.
On those intersections, we define  our function to be $c_0$ ($a_0$ on its dual), hence this causes no conflict either, our function is well-defined. \\
It is clear that $f$ is onto. \\
Let's show that $f$ is not essentially unary.
The fact that
$c_0=f(a_0,c_0,c_0)\neq f(a_0, c_0, a_0)=a_0$
shows that if $f$ depends only on one variable it is the third;
if $f$ depended only its third variable, then its surjectivity would imply that it would be injective on its third variable, which contradicts the fact that
$f(a_0,c_0,a_0)=f(a_0,c_0,a_1)=a_0$. \\
Finally, the reader can easily verify that $f$ is a polymorphism using that $f$ is defined using upsets and downsets. \end{proof}

Let us turn our attention to binary polymorphisms.
Surprisingly, 2-S\l upeckiness of $m\oplus n \oplus k$ depends on $m$, $n$, and $k$:
 if $n$ is small enough compared to $m$ and $k$, then  $m\oplus n \oplus k$ is not 2-S\l upecki, and otherwise, it is.
More precisely, the main result of the section, Theorem \ref{757370}, states that there is a bound $B(m,k)$ for which the poset $m\oplus n \oplus k$ is 2-S\l upecki if and only if $n>B(m,k)+1$.

We start with a standalone, auxiliary result that we will make use of later:

 \begin{lemma}\label{lemma-auto} Let $m,n,k \geq 2$,  let $\bbP=m \oplus n \oplus k$ and let $f:\bbP^2 \rightarrow \bbP$ be a monotone map. 
 Suppose $|B| > 2$ and that $f(B^2) \subseteq B$ has size at least 2. 
 Then there exist $b,b' \in B$ distinct, $\tau_i \in Aut(\bbP)$, $i=0,1$, such that
 \begin{enumerate}
 \item $\tau_i(x)=x$ for all $x \in A \cup C$ and all $i=0,1$; 
 \item  $(\tau_0 \circ f \circ \sigma)(s,s)=s$ for $s \in \{b,b'\}$, where $\sigma(x,y) = (\tau_1(x),y)$.
 \end{enumerate}\end{lemma}
  
  \begin{proof} Notice that $Aut(\bbP)$ contains every transposition of pairs of elements in $B$ (which fixes all other elements in $\bbP$.) We first prove the result in the case where the restriction of $f$ to the diagonal of $B^2$ has image of size at least 2: if $f(b,b)=u$ and $f(b',b')=v$ where $u \neq v$ then choose $\tau_0$ to be any automorphism of $\bbP$ that sends $u$ to $b$ and $v$ to $b'$: then $(\tau_0 \circ f)(s,s)=s$ for $s \in \{b,b'\}$ and we are done. Otherwise, there exist some $z$ such that $f(t,t)=z$ for all $t \in B$, and some $u \neq v$ and $w \neq z$ such that $f(u,v)=w$. Let $\tau_1 \in Aut(\bbP)$  swap $u$ and $v$. Then $(f \circ \sigma)(t,t)= z$ for any $t \not\in \{u,v\}$ and $(f \circ \sigma)(v,v)= f(\tau_1(v),v)=f(u,v)=w$. We conclude with the preceding argument.  
  
  \end{proof}

%%%%%%%%%%%%%%%%%%%%%%%%%%%

%%%%%%%%%%%%%%%%%%%%%%%%%%

Now we define the bound $B(m,k)$ of Theorem \ref{757370}.
%In itself, this definition is cryptic, but, actually, it is the solution of the {\it Combinatorial problem} of the previous explanatory subsection.

\begin{definition} \label{def-bmk} Let $m,k \geq 2$ be integers. We let $\mu(m,k)$ denote the maximum of the function $\alpha\gamma+\beta\delta$ where $\alpha,\beta,\gamma,\delta$ are integers satisfying $0<\alpha,\beta<m$, $0<\gamma,\delta < k$, $(m-\alpha)(m-\beta) \geq m-1$ and $(k-\gamma)(k-\delta) \geq k-1$. Let $B(m,k) = \max(\mu(m,k),mk)$.  
\end{definition}

From the definition and the individual bounds for $\alpha, \beta, \gamma$ and $\delta$, one sees easily that $mk\leq B(m,k) <2mk$ and in particular that the function is well-defined; notice also that $B(m,k) \geq 4$ for all $m,k \geq 2$. 
We have verified numerically that $B(m,k)=mk$ if $m,k\leq 11$. On the other hand, $B(12,12)=145$ as shown by the choice $\alpha=\gamma=9$, and $\beta=\delta=8$, and $B(10,13)=134$. 
%In particular, for any value of $m,k$ such that  $mk \leq 2(m-\sqrt{m-1})(k-\sqrt{k-1})$ the value of $B(m,k)$ will be greater than $mk$.

\begin{comment} Some analysis of this mysterious function $B(m,k)$ follows.

For small $m$ and $k$, $B(m,k)=mk$, but this pattern breaks for large numbers.
Namely, if $m,k\leq 11$, then we have $B(m,k)=mk$ (checked by computer), but, for example, $B(12,12)=145(=12^2+1)$ as shown by the choice $\alpha=\gamma=9$, and $\beta=\delta=8$. 

We state that if $m$ and $k$ tend to infinity, then $B(m,k) \approx 2mk$. 
To verify this, let 
$$\alpha_m=\beta_m=\lfloor m-m^{\frac{2}{3}}\rfloor, \; \text{ and } \;
\gamma_k=\delta_k=\lfloor k-k^{\frac{2}{3}}\rfloor.$$
Then, if $m$ and $k$ are large enough,
$$(m-\alpha_m)(m-\beta_m)\approx m^{\frac{4}{3}} > m-1,  \; \text{ and } \; (k-\gamma_k)(k-\delta_k)\approx k^{\frac{4}{3}} > k-1$$
hold, and we have
$$\alpha_m\gamma_k+\beta_m\delta_k\approx 2mk\left(1-\frac{1}{\sqrt[3]{m}}\right)\left(1-\frac{1}{\sqrt[3]{k}}\right)\approx 2mk.$$
A tighter bound and asymptotes that were obtained by combinatorial heuristics are also included in Appendix \ref{appendix}.

The main result of the section follows.
\end{comment}

\begin{theorem}\label{757370}
The poset $m\oplus n \oplus k$ is 2-S\l upecki if and only if $n>B(m,k)+1$.
\end{theorem}

We split the proof into two separate lemmas (Lemma \ref{Lemma:max_mk} and Lemma \ref{654175}) according to the two directions of the statement.

\begin{definition}  Given any map $f:A^2 \cup C^2 \rightarrow A \cup C$ such that $f(A^2) \subseteq A$ and $f(C^2)\subseteq C$,  for $T \in \{A,C\}$, let $l(T)$ denote the set of all $s \in T$ such that the map  $x\mapsto f(s,x)$ is constant, and let $l(s)$ denote the value of this map. Similarly let $r(T)$ denote the set of all $t \in T$ such that the map  $x\mapsto f(x,t)$ is constant, and let $r(t)$ denote the value of this map. \end{definition}
First, we take on the {\it only if} direction. 
We actually formulate its contrapositive:

\begin{lemma}\label{Lemma:max_mk} If $n \leq B(m,k)+1$, then $m\oplus n \oplus k$ is not  2-S\l upecki.
    
\end{lemma} 
 
  \begin{proof}

With $\bbP=m\oplus n \oplus k$, we will exhibit an essentially binary polymorphism $f: \bbP^2 \to \bbP$.
To do so, we start by defining $f$ partially on $A^2$ and $C^2$.
Let us denote the corresponding restrictions by $f_A$ and $f_C$. \\

{\bf Claim 0.} There exist onto maps $f_A:A^2 \rightarrow A$ and $f_C:C^2 \rightarrow C$ such that there are exactly $B(m,k)$ pairs $(s,t) \in A\times C \cup C \times A$ where  both the maps $x\mapsto f(x,t)$ and $x\mapsto f(s,x)$ are constant, where, for convenience, $f$ stands for the corresponding  $f_A$ or $f_C$. \\
      
{\it Proof of Claim 0.} If $B(m,k) = mk$, then we define $f$ to be the first projection on $A$ and the second projection on $C$. 
Clearly there are exactly $B(m,k)=mk$ pairs $(s,t) \in A\times C$ (and none in $C \times A$) such that both the maps $x\mapsto f(x,t)$ and $x\mapsto f(s,x)$ are constant. 
Otherwise, $B(m,k) = \alpha\gamma+\beta\delta$ with $0<\alpha,\beta<m$, $0<\gamma,\delta < k$. 
By definition of $B(m,k)$ there exists a map $g$ from $\{a_{\alpha},\dots,a_{m-1}\} \times  \{a_{\beta},\dots,a_{m-1}\}$ onto $\{a_1,\dots,a_{m-1}\}$.   
Then define
 $$ f(a_i,a_j)= \begin{cases} a_0 & \mbox{ if $i < \alpha$ or $j < \beta$}, \\
g(a_i,a_j) & \mbox{ otherwise.}
\end{cases}$$

\begin{center}
\end{center}

   \begin{figure}[htb]
\begin{center}
\begin{tabular}{ c|| ccc | ccc |}

$f$ & $a_0$ & $\cdots$ & $a_{\beta-1}$ & $a_{\beta}$ & $\cdots$ & $a_{m-1}$ \\
\hline \hline
$a_0$ &&&&&&\\
$\vdots$ && $a_0$ &&&$a_0$&\\
$a_{\alpha-1}$ &&&&&&\\
\hline
$a_{\alpha}$ &&&&&$g$&\\
$\vdots$ &&$a_0$ & &&\small{with image}&\\
$a_{m-1}$ &&&&&\small{$\{a_1,\dots,a_{m-1}\}$}&\\
\hline
\end{tabular}
 \caption{The ``multiplication table'' for $f$. }
 \end{center}
\end{figure}

%By maximality of $B(m,k)$,  
By definition of $f$ there are exactly $\alpha$ values $a \in A$ such that the map $x\mapsto f(a,x)$ is constant, and $\beta$ values $a \in A$ such that the map $x\mapsto f(x,a)$ is constant.
 Similarly, we can define $f$ on $C^2$ to be onto $C$ and such that there are exactly $\gamma$ values $c \in C$ such that the map $x\mapsto f(x,c)$ is constant and $\delta$ values of $c \in C$ such that the map $x\mapsto f(c,x)$ is constant. 
 In particular, this shows that there are exactly $B(m,k)$ pairs $(s,t) \in (A\times C) \cup (C \times A)$ such that both the map $x\mapsto f(x,t)$ and $x\mapsto f(s,x)$ are constant. \\

Now we will extend $f$ to be a fully defined map on $P^2$.
 Let $b_0 \in B$; by our hypothesis and Claim 0, there exists a map $h:(l(A)\times r(C)) \cup (l(C) \times r(A)) \rightarrow B \setminus \{b_0\}$ which is onto. We now define our map $f$ as follows:
 
 $$f(x,y)= \begin{cases} f_A(x,y)& \mbox{ if $(x,y) \in A^2$}, \\
 					 f_C(x,y)& \mbox{ if $(x,y) \in C^2$}, \\
h(x,y) & \mbox{if $(x,y) \in (l(A)\times r(C)) \cup (l(C) \times r(A))$},\\
l(x) & \mbox{if $x \in l(A)$ and $y \in B$},\\
r(y) & \mbox{if $y \in r(A)$ and $x \in B$},\\
l(x) & \mbox{if $x \in l(C)$ and $y \in B$},\\
r(y) & \mbox{if $y \in r(C)$ and $x \in B$},\\
b_0 & \mbox{otherwise.}

\end{cases}$$

 Clearly $f$ is onto; and since $B(m,k) \geq 4$ implies $l(A)\times r(C) \neq \emptyset$, there is some pair $(a,c) \in A \times C$ such that $f(a,c) \in B$; since $f(a,a) \in A$ and $f(c,c) \in C$, $f$ depends on both variables. Finally we verify that $f$ is monotone. It suffices to verify that $f(x,y) \leq f(u,v)$ where $(u,v)$ covers $(x,y)$ in $\bbP^2$; and it is easy to see that this occurs if and only if $u=x$ and $v$ covers $y$ or $v=y$ and $u$ covers $x$. 
 
 \begin{enumerate}
 \item $(a,a') < (a,b)$: if $a \in l(A)$ then $f(a,a')=f(a,b) = l(a)$; otherwise $f(a,a') \leq b_0 = f(a,b)$. The argument for $(b,a')$ is identical.
 %\item $(a',a) < (b,a)$: if $a \in r(A)$ then $f(a',a)=f(b,a) = r(a)$; otherwise $f(a',a) \leq b_0 = f(b,a)$. 
 \item $(a,b) < (b',b)$: $f(a,b) \in A \cup \{b_0\}$ and $f(b',b) = b_0$. The argument for $(b',a)$ is identical.
 \item $(a,b) < (a,c)$: if $a \not\in l(A)$ then  $f(a,b) = f(a,c)=b_0$; otherwise $f(a,b) \in A$ and $f(a,c) \in B$. The argument for $(b,a) < (c,a)$ is identical.
 \item the remaining cases are dual to the previous ones (replacing $A$ by $C$).  
 \end{enumerate}
 \mbox{}\\
  \end{proof}

It is time to prove the converse (the {\it if} direction of Theorem \ref{757370}).

\begin{lemma} \label{654175}     
  The poset $\bbP=m\oplus n \oplus k$ is 2-S\l upecki if $n > B(m,k)+1$. 
\end{lemma}

 \begin{proof} 
 
Suppose that $\bbP$ admits a binary, not essentially unary onto polymorphism $f$; we want to show that $n \leq B(m,k)+1$. Without loss of generality we may assume that $n \geq 3$. Our goal is to obtain from $f$ a binary, not essentially unary onto polymorphism with a very specific form. \\
 
  \noindent{\bf Claim 1.}  {\em $f(A^2) \supseteq A$, $f(C^2) \supseteq C$   and $f(B^2) \subseteq B$. }\\
 
 \noindent{\em Proof of Claim 1.} The first statement is clear: $f$ is onto, and every element of $\bbP^2$ is above some minimal element which is in $A^2$. The second statement is proved dually. It is then clear that $f(B^2) \subseteq B$ since an element $(b,b') \in B^2$ is above every element of $A^2$ and below every element of $C^2$, so by our first two statements $a \leq f(b,b') \leq c$ for all $a \in A$ and $c \in C$ which shows that $f(b,b') \in B$.\\
 
 \noindent{\bf Claim 2.}  {\em We may assume that $f(C^2)=C$ and $f(A^2)=A$.}\\
 
 \noindent{\em Proof of Claim 2.} Suppose that $f(c_i,c_j) = y \not\in C$ for some $i,j$. Suppose that no other pair is mapped to $y$ by $f$. Then $f(s,t) < y$ for all $s,t \in A \cup B$. Since $y \not\in C$ this means that $f((A\cup B)^2) = A$ (using Claim 1). But $A \cup B$ is connected and $A$ is not, so this is impossible.  Thus 
 $f(u,v)=y$ for some $(u,v)\neq(c_i,c_j)$. Define a map as follows:
 
 $$g(s,t)= \begin{cases} c_1 & \mbox{ if $(s,t)=(c_i,c_j)$}, \\
f(s,t) & \mbox{ otherwise.}
\end{cases}$$
It is clear that $g$ is monotone, onto, and since obviously $f < g$ and the identity is alone in its connected component, we conclude by Lemma \ref{lemma-isolated-unary} that $g$ depends on both variables.  The argument for $A$ is identical. 
  \\
  
   \noindent{\bf Claim 3.} {\em Let $b \in B$. If $b \in f( (A \times B) \cup (B \times A)) \cup  f((C \times B) \cup (B \times C))$ then $b \in f(B^2)$. }\\
 
 \noindent{\em Proof of Claim 3.} Let $(x,y) \in (A \times B) \cup (B \times A)$ with $f(x,y)=b$. Then $(x,y)$  is comparable to a pair $(u,v) \in B^2$ so $f(u,v)=b$ (since $f(u,v) \in B$). Similarly if $(x,y) \in (C \times B) \cup (B \times C)$. \\

  \noindent{\bf Claim 4.} {\em There exists $(x,y) \in (A\times C) \cup (C \times A)$ such that $f(x,y) \in B$.}\\ 

 \noindent{\em Proof of Claim 4.} Otherwise, it follows from Claim 3 that $f(B^2)=B$. By Claim 2 we can then find subsets $A' \subseteq A$, $B' \subseteq B$ and $C' \subseteq C$ of size $m,n,k$ respectively such that $f(A' \cup B' \cup C') = \bbP$. But this gives an embedding $e$ of $\bbP$ into $\bbP^2$ such that $f(e(\bbP))=\bbP$, contradicting Lemma \ref{lemma-embedding}. \\

  \noindent{\bf Claim 5.} {\em Let $b_0 \in f(B^2)$ and let $(a,c) \in A \times C$. 
  \begin{enumerate}
 \item If $f(a,c) \in B\setminus \{b_0\}$, then $(a,c) \in l(A) \times r(C)$;
 \item If $f(c,a) \in B\setminus \{b_0\}$, then $(c,a) \in l(C) \times r(A)$.  
 \end{enumerate}}
 \noindent{\em Proof of Claim 5.} We prove (1) (the proof of (2) is identical). Let $f(b,b')=b_0$ where $b,b',b_0 \in B$. Then $f(b,c) \geq f(a,c),f(b,b')$ so $f(b,c) \in C$. On the  other hand, $f(b,c) \leq f(x,c)$ for all $x \in C$, thus $c \in r(C)$. We have  $f(a,b') \leq f(a,c),f(b,b')$ so $f(a,b') \in A$. On the other hand, $f(a,b') \geq f(a,x)$ for all $x \in A$, thus $a \in l(A)$.\\

 \noindent{\bf Claim 6.} {\em $|f(B^2)|=1$.}\\

 \noindent{\em Proof of Claim 6.} We consider two cases. First suppose that the restrictions $f_A$ and $f_C$ of $f$ to $A^2$ and $C^2$ respectively depend on two different variables, without loss of generality suppose $f_C$ depends on its first variable and $f_A$ depends on its second. This precisely means there exist $a \in A \setminus l(A)$ and $c \in C \setminus r(C)$.  Choose any $b \in B$. Since $f(a,x) \leq f(a,c) \leq f(b,c) \leq f(y,c)$ for all $x \in A$ and all $y \in C$, we conclude that $f(a,c)$ and $f(b,c)$ must be in $B$ (and equal). Now if $b' \in B$ arbitrary, we have $f(b,b') \leq f(b,c)$ and hence they must also be equal. Hence $f(B^2) = \{f(a,c)\}$. 
 
 If the restrictions of $f$ to $A^2$ and $C^2$ do not depend on different variables, it implies they are both essentially unary (and depend on the same variable.) By applying an appropriate $\tau \in Aut(\bbP^2)$, we can then suppose that $f(x,x)=x$ for all $x \in A \cup C$ (notice that Claim 4 guarantees we still obtain an essentially binary operation.)
 Suppose for a contradiction that $|f(B^2)| \geq 2$. By Lemma \ref{lemma-auto}, we may assume that $f(s,s)=s$ for $s \in \{b_1,b_2\}$ (again Claim 4 guarantees we may apply automorphisms and obtain an essentially binary operation). Let $\bbR$ be the subposet of $\bbP$ induced by $A \cup \{b_1,b_2\} \cup C$. 
 It is isomorphic to $m \oplus 2 \oplus k$ and thus is idempotent trivial. 
 There is an obvious retraction $r$ of $\bbP$ onto $\bbR$, that  fixes $A$ and $C$ pointwise and sends $B$ onto $\{b_1,b_2\}$. 
 Consider the restriction of $r \circ f$ to $\bbR^2$: it is idempotent and hence a projection. In particular, for the pair $(x,y)$ from Claim 4,  we will have that $(r \circ f )(x,y) \in \{x,y\}$, but since $f(x,y) \in B$, this is a contradiction. \\
 
  \noindent{\bf Claim 7.} {$|( l(A) \times r(C)) \cup (l(C) \times r(A))|\leq B(m,k)$.}\\

 \noindent{\em Proof of Claim 7.} Let $M=|( l(A) \times r(C)) \cup (l(C) \times r(A))|$. If one of the values in $\{|r(A)|,|r(C)|,|l(A)|,|l(C)|\}$ is equal to 0, then clearly $M \leq mk \leq B(m,k)$. So we may now assume that all 4 values are non-zero.  If both $|r(A)|$ and $|l(A)|$ are non-zero it means that all the maps $x \mapsto f(a,x)$ and $x \mapsto f(x,a)$ that are constant must have the same value, i.e. there exists some $w \in A$ such that $l(a) = r(a') = w$ for all $a \in l(A)$ and all $a' \in r(A)$. By Claim 2 it means that the restriction of $f_A$ to $(A -  l(A)) \times (A- r(A))$ has image $A - \{w\}$, and thus $(m-|l(A)|)(m-|l(B)|) \geq m-1$, and in particular, $|l(A)|,|r(A)| < m$. The same argument shows that the values $|l(C)|,|r(C)|$ also satisfy the required inequalities of Definition \ref{def-bmk}, and thus $M \leq \mu(m,k) \leq B(m,k)$. \\

 We can now conclude the proof. Let $f(B^2)=\{b_0\}$. By Claims 3, 5 and 6, any $b \in B \setminus \{b_0\}$ must be in the image of $( l(A) \times r(C)) \cup (l(C) \times r(A))$. 
 Hence by Claim 7 $|B|-1 \leq |( l(A) \times r(C)) \cup (l(C) \times r(A))|\leq B(m,k)$ and we are done. 
  \end{proof}
  
  \section{Conclusion} \label{sect-conclusion}
  
  We have presented various necessary and sufficient conditions for a reflexive digraph to be S\l upecki. We've shown in particular that there exist posets that triangulate spheres that are idempotent trivial but not 2-S\l upecki; and some idempotent trivial posets that are 2-S\l upecki but not 3-S\l upecki. Here now are some related questions:
  
  \begin{enumerate}
  
  \item Can the condition that the identity is isolated be removed in Lemma \ref{lemma-embedding} ? 
  \item If $\bbG$ is strongly connected and  idempotent trivial, is the identity  alone in its {\em strong} component ? Notice that if this holds, then combining Lemmas \ref{lemma-embedding-strong} and \ref{lemma-not-strong-id-alone} with Lemma \ref{lemma-embedding} gives a positive answer to question (1). 
  \item  What is the algorithmic complexity of recognising S\l upecki (idempotent trivial) digraphs ? 
  
  \item An operation $f$ is {\em conservative} if $f(x_1,\dots,x_n) \in \{x_1,\dots,x_n\}$ for all $x_i$. It is known that a poset of size at least 3 is idempotent trivial if and only if all its idempotent binary polymorphisms are conservative \cite{larose-firstpaper}; hence a poset is idempotent trivial precisely when all its 2-element subsets are pp-definable by gadgets. Does this property hold for reflexive digraphs in general ? Also, although it is not true in general that one can always find gadgets that are trees, is it possible that an algorithm such as singleton arc-consistency (as described in \cite{ChenDalmauGrussien}) might determine if all 2-element subsets are constructible ? If so, this might lead to interesting algorithms for recognising idempotent trivial posets. 
  
  \item the bound $B(m,k)$ we introduced is intriguing from a combinatorial standpoint. Investigate.
  
  \item Are there posets (or more generally digraphs) that are not S\l upecki but $n$-S\l upecki for arbitrary large $n$ ? What is the relationship to the size of the digraph ? 
  
  \item The idempotent trivial posets of height 1 are known \cite{coro}; they are the height 1 posets that are connected and such that no element has a unique upper nor unique lower cover. Are they all S\l upecki ? 
 \end{enumerate}

 \bibliographystyle{plain}

\end{document}